\documentclass[12pt]{article}
\newcommand{\authorfootnotes}{\renewcommand\thefootnote{\@fnsymbol\c@footnote}}%
\usepackage{amssymb,amsmath, amsfonts, amsthm}

\usepackage{graphicx}
\usepackage{pgfplots}
\usepackage{tikz}
\usepackage{caption}
\usepackage{booktabs}
\usepackage{array}
\usepackage{verbatim}
\usepackage{subfig}
\usepackage{geometry}
\usepackage{fancyhdr}

\newcommand{\N}{\mathbb{N}}

\newcommand{\cC}{\mathcal{C}}
\newcommand{\F}{\mathcal{F}}
\newcommand{\G}{\mathcal{G}}
\newcommand{\cH}{\mathcal{H}}

\newcommand{\vertex}{\node[vertex]}
\tikzstyle{vertex}=[circle, draw, inner sep=0pt, minimum size=6pt]
\newtheorem{theorem}{Theorem}
\newtheorem{lemma}{Lemma}

\newtheorem{prop}{Proposition}

\newtheorem{problem}{Problem}

\newtheorem{definition}{Definition}

\newcommand{\gedge}{\gamma'}
\newcommand{\Gedge}{\Gamma'}

\begin{document}

\title{On well-edge-dominated graphs}

\author{$^1$Sarah E. Anderson, $^2$Kirsti Kuenzel and $^3$Douglas F. Rall \\
\\
$^1$Department of Mathematics \\
University of St. Thomas \\
St. Paul, Minnesota  USA\\
\small \tt Email: ande1298@stthomas.edu  \\
\\
$^2$Department of Mathematics \\
Trinity College \\
Hartford, Connecticut  USA \\
\small \tt Email: kwashmath@gmail.com\\
\\
$^3$Department of Mathematics \\
Furman University \\
Greenville, SC, USA\\
\small \tt Email: doug.rall@furman.edu}

\date{}
\maketitle

\begin{abstract}
A graph is said to be well-edge-dominated if all its minimal edge dominating sets are minimum. It is known that every well-edge-dominated graph $G$ is also equimatchable, meaning that every maximal matching in $G$ is maximum. In this paper, we show that if $G$ is a connected, triangle-free, nonbipartite, well-edge-dominated graph, then $G$ is one of three graphs. We also characterize the well-edge-dominated split graphs and Cartesian products. In particular, we show that a connected Cartesian product $G\Box H$ is well-edge-dominated, where $G$ and $H$ have order at least $2$, if and only if $G\Box H = K_2 \Box K_2$.

\end{abstract}

\section{Introduction}

A set $F$ of edges in a graph $G$ is an \emph{edge dominating set} if every edge of $G$ that is not in $F$ is adjacent to at least one edge in $F$.  Mitchell and
Hedetniemi~\cite{mh-1977} initiated the study of edge domination by presenting a linear algorithm that finds a smallest edge dominating set in a tree.  Yannakakis and
Gavril~\cite{yg-1980} showed that it is NP hard to find an edge dominating set of minimum size even when restricted to planar graphs or subcubic bipartite graphs.  See~\cite{bp-2008, hk-1993,
hc-1995} for additional results on the complexity of finding a minimum edge dominating set.  The set consisting of all the vertices that are incident with at least
one edge in a minimum edge dominating set is a vertex  dominating set in a nontrivial connected graph.  It follows that the (ordinary) domination number of such a graph is at most twice  the size of its smallest edge dominating set.  Senthilkumar, Venkatakrishna and Kumar~\cite{svk-2020} characterized the trees that achieve equality of these numbers,
and Baste, F\"{u}rst, Henning, Mohr and Rautenbach~\cite{bfhmr-2020} gave an improvement of this relationship when the graph is regular.  They conjectured that the domination number is at most the edge domination number in every regular graph.  Klostermeyer and Yeo~\cite{ky-2015} investigated edge domination in grid graphs.
See~\cite{av-1998, c-2010, t-1993} for other problems involving edge domination.

The graphs for which all maximal matchings have the same cardinality were first studied independently by Lewin~\cite{l-1974} and Meng~\cite{m-1974} in 1974.
These two authors presented different characterizations of this class of graphs that have come to be known as \emph{equimatchable}.
Lesk, Plummer and Pulleyblank~\cite{lpp-1983} gave a characterization of equimatchable graphs that gave rise to a polynomial time algorithm for recognizing
membership in this class of graphs.  Since then the structure of several subclasses of equimatchable graphs have been investigated.  Frendrup, Hartnell
and Vestergaard~\cite{fhv-2010} proved that a connected equimatchable graph with no cycles of length less than $5$ is either  a $5$-cycle, a $7$-cycle or
belongs to the family $\cC$ that contains $K_2$ and all the bipartite graphs one of whose partite sets consists of all its support vertices.
B\"{u}y\"{u}k\c{c}olak, G\"{o}z\"{u}pek and S.~\"{O}zkan~\cite{bgo-2021} provided a complete structural characterization of the connected, triangle-free equimatchable
graphs that are not bipartite.

If $M$ is a maximal matching  in a graph $G$, then every edge not in $M$ is adjacent to at least one edge in $M$. That is, $M$ is
an edge dominating set of $G$.  A maximal matching in $G$ corresponds to a maximal independent set in the line graph of $G$.  Since line graphs are claw-free, the independent
domination number (the smallest cardinality among the independent dominating sets) of the line graph equals its domination number.  Translating this back to
the original graph, it means the size of a smallest maximal matching in $G$ and the edge domination number of $G$ coincide.  Frendrup, et al. also proved
in~\cite{fhv-2010} that every graph in $\cC$ has the additional property that all of its minimal edge dominating sets have the same cardinality.
In this paper we study graphs that have this latter property and call them \emph{well-edge-dominated}.  In particular, we completely characterize three classes of connected well-edge-dominated graphs.

Our main result on triangle-free, nonbipartite well-edge-dominated graphs is the following result, which is proved in Section~\ref{sec:K3-free}.
We use the characterization, mentioned above, by B\"{u}y\"{u}k\c{c}olak, et al.~\cite{bgo-2021}, of the equimatchable graphs satisfying the hypothesis of Theorem~\ref{thm:K3-freeWED} and determine which of these belong to the smaller class of well-edge-dominated graphs.
The graph $H^*$ is defined in Section~\ref{sec:K3-free}.

\begin{theorem} \label{thm:K3-freeWED}
If $G$ is a connected, nonbipartite, well-edge-dominated graph of girth at least $4$, then $G \in \{C_5, C_7, H^*\}$.
\end{theorem}

A graph is a \emph{split graph} if its vertex set admits a partition into two sets, one of which is independent and the other which induces a complete graph.
We show that a connected split graph is well-edge-dominated if and only if it is a star, a complete graph of order at most $4$, a graph obtained
from $C_5$ by adding two adjacent chords, or belongs to one of two families of graphs constructed from $K_4$.  These are defined in Section~\ref{sec:SplitWED}.

In Section~\ref{sec:CartesianWED} we finish by showing that $C_4$ is the only nontrivial, connected, well-edge-dominated Cartesian product.

\begin{theorem} \label{thm:cpWED}
If $G$ and $H$ are two connected, nontrivial graphs, then $G\Box H$ is well-edge-dominated if and only if $G\Box H = K_2 \Box K_2$.
\end{theorem}

\section{Preliminaries}

All the graphs considered in this paper are simple and have  finite order.  Let $G$ be a graph with vertex set $V(G)$ and edge set $E(G)$.  We write $n(G)=|V(G)|$.
If $n(G) \ge 2$, then $G$ is \emph{nontrivial}.
For a positive integer $k$ the set of positive integers no larger than $k$ is denoted $[k]$.  Although edges are $2$-element subsets of vertices, for simplicity we will shorten the notation of an edge $\{u,v\}$ to $uv$.  If $X \subseteq E(G)$, then $G-X$ is the graph with vertex set $V(G)$ and edge set $E(G)-X$.  For graphs $G$ and $H$, the Cartesian product $G\Box H$ has vertex set
$\{(g,h)\,:\, g\in V(G), h \in V(H)\}$.  Two vertices $(g_1,h_1)$ and $(g_2,h_2)$ are adjacent in $G\Box H$ if either $g_1=g_2$ and $h_1h_2\in E(H)$ or $h_1=h_2$ and
$g_1g_2 \in E(G)$.  For $g \in V(G)$ the \emph{$H$-fiber} $^gH$ is the subgraph of $G\Box H$ induced by the set $\{(g,h)\,:\,h\in V(H)\}$.  Similarly, the \emph{$G$-fiber}
$G^h$ for a given vertex $h \in V(H)$ denotes the subgraph induced by $\{(g,h)\,:\,g\in V(G)\}$.  Note that $^gH$ is isomorphic to $H$ and $G^h$ is isomorphic to $G$.

Two distinct edges $e$ and $f$ in a graph $G$ are \emph{adjacent} if $e \cap f \neq \emptyset$ and are \emph{independent} if $e \cap f = \emptyset$.  A vertex $x$ of $G$ is \emph{incident} to an edge $e$ if $x \in e$.   If $X \subseteq E(G)$, then the set of vertices \emph{covered} by $X$ is denoted by $S(X)$ and is defined by $S(X)=\{u \in V(G) \,:\, u \text{ is incident to an edge in } X\}$.  Let  $f\in E(G)$ and let $F \subseteq E(G)$.  The \emph{closed edge neighborhood} of $f$ is the set $N_e[f]$ consisting of $f$ together with all edges in $G$ that are adjacent to $f$.
The \emph{closed edge neighborhood} of $F$ is the set $N_e[F]$ defined by $N_e[F]=\cup_{f\in F}N_e[f]$.  Let $f \in F$.  The edge $f$ is said to \emph{dominate} the set $N_e[f]$. An edge $g$ is called a \emph{private edge neighbor} of $f$ with respect to $F$ if $g \in N_e[f]-N_e[F-\{f\}]$.
If $N_e[F]= E(G)$, then $F$ is called an \emph{edge dominating set} of $G$.  The \emph{edge domination number} of $G$, denoted by $\gedge(G)$, is the smallest cardinality of an edge dominating set in $G$, and the \emph{upper edge domination
number} of $G$ is the largest cardinality, $\Gedge(G)$, of a minimal edge dominating set.  A \emph{matching} in $G$ is a set of independent edges.  The \emph{matching number}
of $G$, denoted $\alpha'(G)$, is the number of edges in a matching of largest cardinality in $G$, while the \emph{lower matching number} is the number of edges, denoted by
$i'(G)$, in a smallest maximal matching.  Any maximal matching $M$ in $G$ is clearly a minimal edge dominating set of $G$, which gives
\[\gedge(G)\le i'(G) \le \alpha'(G) \le \Gedge(G)\,.\]
A graph $G$ is called \emph{equimatchable} if $i'(G) = \alpha'(G)$ and is called \emph{well-edge-dominated} if $\gedge(G)=\Gedge(G)$.  Using the inequality
above it is clear that the class of well-edge-dominated graphs is a subclass of the equimatchable graphs.

It is clear that a graph is well-edge-dominated if and only if each of its components is well-edge-dominated.  We use this
fact throughout the paper together with the following lemmas.

A very useful tool in our study of well-edge-dominated graphs is the following lemma, which is the ``edge version'' of a fact used by Finbow,
Hartnell and Nowakowski in~\cite{fhn-1988}.  It follows from the fact that $M \cup D_1$ and $M \cup D_2$ are both minimal edge dominating sets of $G$
for any matching $M$ and any pair $D_1$ and $D_2$ of minimal edge dominating sets of the graph $G-N_e[M]$.

\begin{lemma} \label{lem:matchremove}
If $G$ is a well-edge-dominated graph and $M$ is any matching in $G$, then $G-N_e[M]$ is well-edge-dominated.
\end{lemma}

The next two results show that several common graph families contain only a small number of well-edge-dominated graphs.

\begin{lemma}\label{lem:Kn}
A complete graph of order $n$ is well-edge-dominated if and only if $n \le 4$.
\end{lemma}

\begin{proof}
Using the definition we see that the complete graphs of order at most $4$ are well-edge-dominated.  For the converse suppose $n \ge 5$.
 Label the vertices of $K_{n}$ as $1, \dots, n$ and consider the set $D = \{12, 13, \dots, 1(n-1)\}$. We claim that $D$ is a minimal edge dominating set. Indeed, $D - \{1j\}$ is not an edge dominating set since $jn$ is not adjacent to any edge in $D - \{1j\}$. Therefore, $D$ is in fact a minimal edge dominating set of cardinality $n-2$ where $n\ge 5$. On the other hand, we can choose a matching of $K_n$ of cardinality $\left\lfloor \frac{n}{2}\right\rfloor$. Note that $n-2 > \frac{n}{2}$ when $n \ge 5$. Thus, $K_n$ is not well-edge-dominated.
\end{proof}

Any star is well-edge-dominated and we show in Theorem~\ref{thm:wedrm} that $K_{n,n}$ is well-edge-dominated for any $n \ge 1$. No
other complete bipartite graph is well-edge-dominated as the following lemma shows.

\begin{lemma} \label{lemma:nonwed}
If $2 \le r < s$, then $K_{r, s}$ is not well-edge-dominated.
\end{lemma}

\begin{proof}
Assume $2 \le r<s$ and write the partite sets of $K_{r, s}$ as $\{x_1, \dots, x_r\}$ and $\{y_1, \dots, y_s\}$. Note that $\{x_1y_1, \ldots, x_1y_s\}$
and $\{x_1y_1, x_2y_2, \ldots, x_ry_r\}$ are two minimal edge dominating sets of different cardinalities.  Therefore, $K_{r,s}$ is not well-edge-dominated.
\end{proof}

\section{Randomly matchable graphs}
A graph is said to be {\it randomly matchable} if every maximal matching is a perfect matching.  That is, a randomly matchable graph is
an equimatchable graph whose matching number is half its order.  Sumner~\cite{s-1979} determined all the randomly matchable graphs.

\begin{theorem}{\rm (\cite{s-1979})} \label{thm:rm}
A connected graph is randomly matchable if and only if it is isomorphic to $K_{2n}$ or $K_{n,n}$ for $n \ge 1$.
\end{theorem}

Using Theorem~\ref{thm:rm} we can now show which randomly matchable graphs are well-edge-dominated.

\begin{theorem}\label{thm:wedrm}
A connected graph $G$ containing a perfect matching is well-edge-dominated if and only if $G=K_4$ or $G = K_{n,n}$ for $n \ge 1$.
\end{theorem}

\begin{proof}
Suppose first that $G$ contains a perfect matching and is well-edge-dominated. It follows that $G$ is equimatchable and every maximal matching is of size $n(G)/2$. Therefore, $G$ is randomly matchable and by Theorem~\ref{thm:rm}, $G = K_{2n}$ or $G= K_{n,n}$ for $n \ge 1$.  By Lemma~\ref{lem:Kn}, $K_{2n}$ for $n \ge 3$ is not well-edge-dominated. It follows that $G = K_4$ or $G = K_{n,n}$ for $n \ge 1$.

In the other direction, suppose $G=K_4$ or $G=K_{n,n}$ for $n \ge 1$. One can easily verify that $K_4$ is well-edge-dominated. Therefore, we shall assume $G = K_{n,n}$ and let $A$ and $B$ be the partite sets of $G$. We show that $G$ is well-edge-dominated. Let $D$ be an edge dominating set of $G$. Suppose $D$ does not cover $a \in A$ and $b \in B$. Then $ab$ is not dominated by $D$, which is a contradiction. Thus, we may assume $D$ covers $A$ which implies $|D| \ge n$. Suppose that $|D| > n$. It follows that some vertex of $A$ is incident to two edges in $D$, say $e$ and $f$. Note that $D - \{e\}$ is an edge dominating set of $G$ since $D - \{e\}$ covers $A$ and every edge of $G$ is incident to exactly one vertex of $A$. Thus, $|D| = n$ and $G$ is well-edge-dominated.

\end{proof}

\section{Triangle-free nonbipartite graphs} \label{sec:K3-free}

In this section we prove there are only three nonbipartite, triangle-free, connected,  well-edge-dominated graphs.  These
three graphs are the $5$-cycle, the $7$-cycle and $H^*$, which is depicted in Figure~\ref{fig:triangle-wed}.
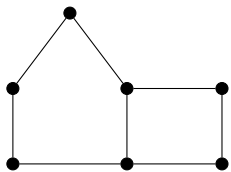
\begin{figure}[ht!]
\begin{center}
\begin{tikzpicture}[auto]
    % Place nodes
	
	\vertex (21) at (-5,-4) [label=above:$$,scale=.75pt,fill=black]{};
	\vertex (22) at (-5.75,-5) [label=above:$$,scale=.75pt,fill=black]{};
	\vertex (23) at (-4.25, -5) [label=above:$$,scale=.75pt,fill=black]{};
	\vertex (24) at (-3,-5) [label=above:$$,scale=.75pt,fill=black]{};
	\vertex (25) at (-5.75, -6) [label=right:$$,scale=.75pt,fill=black]{};
	\vertex (26) at (-4.25,-6) [label=below:$$,scale=.75pt,fill=black]{};
	\vertex (27) at (-3, -6) [label=left:$$,scale=.75pt,fill=black]{};

	\path
		
		(21) edge (22)
		(21) edge (23)
		(22) edge (25)
		(23) edge (26)
		(25) edge (26)
		(24) edge (27)
		(27) edge (26)
		(23) edge (24);
\end{tikzpicture}
\end{center}
\caption{The graph $H^*$}
\label{fig:triangle-wed}
\end{figure}

We will use the structural characterization of the class of triangle-free, equimatchable graphs in the recent paper of B\"{u}y\"{u}k\c{c}olak, G\"{o}z\"{u}pek and \"{O}zkan~\cite{bgo-2021}.  To describe their characterization, they defined several graph families using the following notation. Let $H$ be a graph on $k$ vertices $v_1, v_2, \dots, v_k$ and let $m_1, m_2, \dots, m_k$ be nonnegative integers. Then $H(m_1, m_2, \dots, m_k)$ denotes the graph obtained from $H$ by repeatedly replacing each vertex $v_i$ with an independent set of $m_i$ vertices, each of which has the same neighborhood as $v_i$.  For example, using the graph $G^*$ in Figure~\ref{fig:G*}, we see that
$G^*(1,1,1,0,1,1,1,1,0,0,0)=C_7$ and $G^*(2,0,0,0,3,0,0,0,2,3,0)=K_{4,6}$.

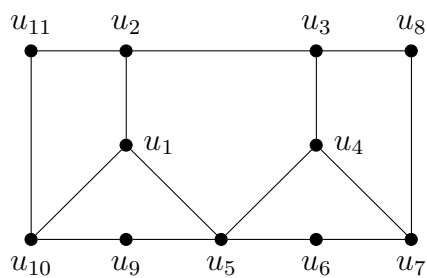
\begin{figure}[h!]
\begin{center}
\begin{tikzpicture}[scale=1.25]

	\vertex (8) at (5,0) [fill, scale=.75,label=below:$u_{10}$] {};
	\vertex (9) at (5,2) [fill, scale=.75,label=above:$u_{11}$] {};
	\vertex (10) at (6,2) [fill, scale=.75,label=above:$u_2$] {};
	\vertex (11) at (8,2) [fill, scale=.75,label=above:$u_3$] {};
	\vertex (12) at (9,2) [fill, scale=.75,label=above:$u_8$] {};
	\vertex (13) at (9,0) [fill, scale=.75,label=below:$u_7$] {};
	\vertex (14) at (8,0) [fill, scale=.75,label=below:$u_6$] {};
	\vertex (15) at (7,0) [fill, scale=.75,label=below:$u_5$] {};
	\vertex (16) at (6,0) [fill, scale=.75,label=below:$u_9$] {};
	\vertex (17) at (6, 1) [fill, scale=.75,label=right:$u_1$] {};
	\vertex (18) at (8,1) [fill, scale=.75,label=right:$u_4$] {};
	
	\path
	
	(8) edge (9)
	(9) edge (10)
	(10) edge (11)
	(11) edge (12)
	(12) edge (13)
	(13) edge (14)
	(14) edge (15)
	(15) edge (16)
	(16) edge (8)
	(17) edge (8)
	(17) edge (10)
	(17) edge (15)
	(18) edge (11)
	(18) edge (13)
	(18) edge (15)

	;
\end{tikzpicture}
\end{center}
\caption{ The graph $G^*$}
\label{fig:G*}
\end{figure}

The following definition was made in~\cite{bgo-2021}.

\begin{definition} {\rm (\cite{bgo-2021})} \label{def:F}
Let $G^*$  be the graph in Figure~\ref{fig:G*} and let $\F$ be the union of the following six graph families.
\begin{enumerate}
\item $\F_{11} = \{G^*(1, 1, 1, 1, 1, n, n, 0, 0, 0, 0): n \ge 1\}$
\item $\F_{12} = \{G^*(1,1,1,0,1,n+1,n+1,1,0,0,0): n\ge 1\}$
\item $\F_{21} = \{G^*(1,1,1,n-r-s+1,1,r,n,s,0,0,0): n-1\ge r\ge 1, n-1\ge s\ge 1, n\ge r+s\}$
\item $\F_{22} = \{G^*(1,1,1,n-r-s,1,r+1,n+1,s+1,0,0,0): n-1\ge r\ge 1, n-1\ge s\ge 1, n\ge r+s\}$
\item $\F_3 = \{G^*(1,1,r+1,s+1,1,0,n-s,n-r,0,0,0): n-1\ge r\ge 1, n-1\ge s\ge 1\}$
\item $\F_4 = \{G^*(r+1,n+1,s+1,1,1,0,0,0,0,0,n-r-s): n-1\ge r\ge 1, n-1\ge s\ge 1, n\ge r+s\}$
\end{enumerate}
\end{definition}

By analyzing each of the six families of equimatchable graphs listed above, we determine all the  well-edge-dominated graphs in $\F$.

\begin{prop} \label{prop:Fnonwed}
If $G \in \F$ is well-edge-dominated, then $G = H^*$.
\end{prop}

\begin{proof}
Throughout this proof when considering a graph from one of these six families of graphs we will always assume the variables (that is, whichever of
$n,r$ and $s$ are used) satisfy the conditions in Definition~\ref{def:F} for that particular family.

First, let  $G=G^*(1, 1, 1, 1, 1, n, n, 0, 0, 0, 0) \in \F_{11}$. Note first that if $n=1$, then $G=H^*$  depicted in Figure~\ref{fig:triangle-wed}.
It is straightforward to show that $H^*$ is well-edge-dominated.  Suppose $n \ge 3$ and let $\{x_1, \dots, x_n\}$ be the set of vertices that
replace $u_6$ and let $\{y_1, \dots, y_n\}$
be the set of vertices that replace $u_7$. Note that $K_{n-1, n}$ is a component of   $G-N_e[\{x_1u_5, u_3u_4\}]$.
By Lemma~\ref{lemma:nonwed}, we infer that $G$ is not well-edge-dominated.   Therefore, we shall assume $n=2$.  Now, $\{u_1u_2,u_3u_4\}$ is a matching,
and $K_{2,3}$ is a component of $G-N_e[\{u_1u_2,u_3u_4\}]$. By Lemma~\ref{lem:matchremove} and Lemma~\ref{lemma:nonwed}, it follows that $G$ is not well-edge-dominated.

Next, let $G = G^*(1,1,1,0,1,n+1,n+1,1,0,0,0) \in \F_{12}$.  Let $\{x_1, \dots, x_{n+1}\}$ be the set of vertices that replace $u_6$ and
let $\{y_1, \dots, y_{n+1}\}$ be the set of vertices that replace $u_7$. Suppose first that $n \ge 2$.   Note that $K_{n, n+1}$ is a component of $G-N_e[\{x_1u_5, u_3u_8\}]$.  Since
$K_{n,n+1}$ is not well-edge-dominated by Lemma~\ref{lemma:nonwed}, it follows from Lemma~\ref{lem:matchremove} that $G$ is not well-edge-dominated. Therefore, we shall assume $n=1$. In this case, both $\{x_1y_1, x_2y_2, u_1u_5, u_3u_8\}$ and $\{x_1y_1, x_2y_1, u_8y_1, u_1u_5, u_2u_3\}$ are both minimal edge dominating sets, and hence $G$ is not well-edge-dominated.

Next, let $G \in \F_{21} \cup \F_{22} \cup \F_4$. Note that $n \ge 2$ for every such $G$.  Suppose $G = G^*(1,1,1,n-r-s+1,1,r,n,s,0,0,0) \in \F_{21}$. Note that $G - N_e[\{u_2u_3, u_1u_5\}]=K_{n, n+1}$. If $G = G^*(1,1,1,n-r-s,1,r+1,n+1,s+1,0,0,0) \in \F_{22}$, then $G - N_e[\{u_1u_5,u_2u_3\}] = K_{n+1, n+2}$. If
$G = G^*(r+1,n+1,s+1,1,1,0,0,0,0,0,n-r-s) \in \F_4$, then $G - N_e[\{u_4u_5\}]=K_{n+1, n+2}$. Therefore, for every $G \in \F_{21} \cup \F_{22} \cup F_4$,
we see by Lemmas~\ref{lem:matchremove} and~\ref{lemma:nonwed} that  $G$ is not well-edge-dominated.

Finally, assume $G \in \F_3$ and $G = G^*(1,1,r+1,s+1,1,0,n-s,n-r,0,0,0)$. Let $\{x_1,\ldots,x_{s+1}\}$ be the set of vertices that replace $u_4$.  The complete bipartite
graph $K_{n-s+r+1,n-r+s}$ is a component of $G-N_e[\{u_1u_2,u_5x_1\}]$.  Observe that $n-s+r+1 \neq n-r+s$ for otherwise $2r+1=2s$, which is not possible.  Furthermore, using
the conditions $n-1\ge r\ge 1$  and  $n-1\ge s\ge 1$ we see that $n-s+r+1 \ge 3$ and $n-r+s \ge 2$.  It follows by Lemma~\ref{lemma:nonwed} that $G$ is not well-edge-dominated.
\end{proof}

\begin{definition} {\rm (\cite{bgo-2021})} \label{def:G}
Let $G^*$  be the graph in Figure~\ref{fig:G*} and let $\G$ be the union of the following seven graph families.
\begin{enumerate}
\item $\G_{11} = \{G^*(m+1,m+1,1,0,1,1,n+1,n+1,0,0,0): n\ge 1, m\ge 1\}$
\item $\G_{12} = \{G^*(m+1,m+1,1,n-r-s,1,r+1,n+1,s+1,0,0,0): m \ge 1, n-1 \ge r \ge 1, n-1 \ge s \ge 1, n \ge r+s\}$
\item $\G_{21} = \{G^*(1,1,1,n-r-s+1,1,r,n,s,0,m,m):  m\ge 1, n-1\ge r\ge 1, n-1\ge s\ge 1, n \ge r+s\}$
\item$\G_{22} = \{G^*(1,1,r+1,s+1,1,0,n-s,n-r,0,m,m):  m\ge 1, n-1\ge r\ge 1, n-1\ge s\ge 1\}$
\item$\G_{23} = \{G^*(r+1,n+1,s+1,1,1,m,m,0,0,0,n-r-s):  m\ge 1, n-1 \ge r \ge 1, n-1 \ge s \ge 1, n \ge r+s\}$
\item $\G_{31} = \{G^*(m-k-\ell+1,1,1,n-r-s+1,1,r,n,s,\ell,m,k): n-1 \ge r \ge 1, n-1 \ge s \ge 1, n \ge r+s, m-1 \ge \ell \ge 1, m-1 \ge k \ge 1,
               m \ge k+\ell\}$
\item $\G_{32} = \{G^*(k+1,\ell +1,1,n-r-s+1,1,r,n,s,0,m-\ell,m-k):  n-1 \ge r \ge 1, n-1 \ge s \ge 1, n \ge r+s, m-1 \ge \ell \ge 1, m-1 \ge k \ge 1,
               m \ge k+\ell\}$
\end{enumerate}
\end{definition}

As we did in Proposition~\ref{prop:Fnonwed}, an analysis of all the graphs in $\G$ will show that no such graph is well-edge-dominated.

\begin{prop} \label{prop:Gnonwed}
If $G \in \G$, then $G$ is not well-edge-dominated.
\end{prop}

\begin{proof}
Throughout this proof when considering a graph from one of these seven families of graphs we will always assume the variables (that is, whichever of
$n,m,r,s,k$ and $\ell$ are used) satisfy the conditions in Definition~\ref{def:G} for that particular family.

First, suppose $G \in \G_{11} \cup \G_{12}$.   Let $\{x_1, \dots, x_{m+1}\}$ be the set of vertices that replace $u_2$ and let $\{y_1, \dots, y_{m+1}\}$
be the set of vertices that replace $u_1$.  If $G=G^*(m+1,m+1,1,0,1,1,n+1,n+1,0,0,0)\in \G_{11}$, then
$K_{n+1,n+2}$ is a component of $G-N_e[\{x_1u_3,y_1u_5\}]$.  If $G=G^*(m+1,m+1,1,n-r-s,1,r+1,n+1,s+1,0,0,0)\in \G_{12}$, then $K_{n+1,n+2}$ is a component of $G-N_e[\{x_1u_3,y_1u_5\}]$.  Since $n+1\ge 2$, it follows from Lemmas~\ref{lem:matchremove} and~\ref{lemma:nonwed} in both cases that $G$ is not well-edge-dominated.

Next, suppose $G = G^*(1,1,1,n-r-s+1,1,r,n,s,0,m,m) \in \G_{21}$. Note that this implies $n\ge 2$ and  $G- N_e[\{u_1u_5, u_2u_3\}]$
contains the component $K_{n, n+1}$.  By Lemmas~\ref{lem:matchremove} and~\ref{lemma:nonwed} we infer that $G$ is not well-edge-dominated.

Next, suppose $G = G^*(1,1,r+1,s+1,1,0,n-s,n-r,0,m,m) \in \G_{22}$. Let $\{x_1, \ldots, x_m\}$ be the set of vertices that replace
$u_{11}$ and let $\{y_1, \ldots, y_{r+1}\}$ be the set of vertices that replace $u_3$.  The complete bipartite graph $K_{n-r+s+1,n-s+r+1}$ is a component of $G-N_e[\{x_1u_2,u_1u_5\}]$.  Note that $n-r+s+1 \ge s+2\ge 3$ and $n-s+r+1 \ge r+2 \ge 3$.
If $n-r+s+1 \neq n-s+r+1$, then $K_{n-r+s+1,n-s+r+1}$ is not well-edge-dominated by Lemma~\ref{lemma:nonwed}.  On the other hand, if $n-r+s+1 = n-s+r+1$, then
$G-N_e[\{u_2y_1,u_1u_5\}]$ has a component isomorphic to $K_{n-r+s+1,n-s+r}$, which is not well-edge-dominated.  Again by Lemmas~\ref{lem:matchremove} and~\ref{lemma:nonwed}
we conclude that $G$ is not well-edge-dominated.

Next, suppose $G = G^*(r+1,n+1,s+1,1,1,m,m,0,0,0,n-r-s)\in \G_{23}$.  The graph $K_{n+1,n+2}$ is a component of $G-N_e[u_4u_5]$.  Using Lemmas~\ref{lem:matchremove} and~\ref{lemma:nonwed}  we infer that $G$ is not well-edge-dominated.

Next, suppose $G = G^*(m-k-\ell+1,1,1,n-r-s+1,1,r,n,s,\ell,m,k) \in \G_{31}$. Let $\{x_1, \ldots, x_{m-k-\ell+1}\}$ be the set of vertices that
replace $u_1$.  Note that $n \ge 2$ and that $K_{n,n+1}$ is a component of $G-N_e[\{x_1u_5,u_2u_3\}]$.  By Lemmas~\ref{lem:matchremove} and~\ref{lemma:nonwed},
this implies that $G$ is not well-edge-dominated.

Finally, suppose $G = G^*(k+1,\ell +1,1,n-r-s+1,1,r,n,s,0,m-\ell,m-k) \in \G_{32}$.  Note that $n \ge 2$.  Let $\{x_1, \dots, x_{\ell+1}\}$ be the set of vertices
that replace $u_2$ and let $\{y_1, \dots, y_{k+1}\}$ be the set of vertices that replace $u_1$.  Since $K_{n,n+1}$ is a component of $G-N_e[\{x_1u_3,y_1u_5\}]$,
we conclude by Lemmas~\ref{lem:matchremove} and~\ref{lemma:nonwed} that $G$ is not well-edge-dominated.

\end{proof}

\vskip5mm

\noindent \textbf{Theorem~\ref{thm:K3-freeWED}} \emph{
If $G$ is a connected, nonbipartite, well-edge-dominated graph of girth at least $4$, then $G \in \{C_5, C_7, H^*\}$.
}

\begin{proof}
It is straightforward to check that every graph in $\F \cup \G$ is connected, has girth $4$ but is not bipartite.  If we consider only nonbipartite graphs, then the
main result of B\"{u}y\"{u}k\c{c}olak, et.~al~\cite[Theorem 36]{bgo-2021} states that a graph $G$ is a connected, nonbipartite, triangle-free equimatchable graph if
and only if $G \in \F \cup \G \cup \{C_5,C_7\}$.
Applying Proposition~\ref{prop:Fnonwed} and Proposition~\ref{prop:Gnonwed} completes the proof.
\end{proof}

\section{Split graphs} \label{sec:SplitWED}

Recall that a graph is a split graph if its vertex set can be partitioned into an independent set and a set that induces a complete graph.  In this section we
prove a complete characterization of the family of split graphs that are well-edge-dominated.  We will use the following definitions.
Let $\cH_1$ be the family of graphs obtained by appending any finite number of leaves to a single vertex of $K_4$ and let $\cH_2$ be the family of  graphs obtained from $K_4$ by removing any edge $uv$  and appending  at least one leaf to $u$.  Let $H_3$ be the graph of order $5$ obtained from $K_4-e$ by adding a new vertex adjacent to one of the vertices of degree $2$ and one of the vertices of degree $3$.
\begin{lemma} \label{lem:WEDsplitgraphs}
If $G\in \{K_2, K_3, K_4, H_3\} \cup \cH_1 \cup \cH_2 \cup \{K_{1,n} : n \in \N\}$, then $G$ is well-edge-dominated.
\end{lemma}
\begin{proof}
By Lemma~\ref{lem:Kn}, $K_2$, $K_3$, and $K_4$ are well-edge-dominated. It is easy to see that every minimal edge dominating set of a nontrivial star $K_{1,n}$
consists of exactly one edge.  Therefore, $K_{1,n}$ is well-edge-dominated.  It is straightforward to check that all minimal edge dominating sets of $H_3$ have
cardinality $2$.

Next, assume $G \in \cH_1$. Suppose the vertices $v_1,v_2,v_3$ and $v_4$ of $G$ induce a complete graph and $v_1$ is the support vertex. Let $D$ be a minimal edge dominating set of $G$. First assume that $D$ contains an edge, say $v_1w$, where $w$ is a leaf. Note that $D$ cannot contain more than one edge incident with $v_1$ since $D$ is minimal. The only edges not dominated by $v_1w$ are $v_2v_3, v_2v_4$ and $v_3v_4$. Exactly one of those edges must be in $D$ in order for it to be a minimal edge dominating set. Thus, $|D| = 2$. Next, assume $D$ does not contain an edge incident to a leaf. Then $D \cap \{v_1v_2, v_1v_3, v_1v_4\} \neq \emptyset$. Without loss of generality, assume $v_1v_2 \in D$. The only edge of $G$ not dominated by $v_1v_2$ is $v_3v_4$, so by minimality $|D|=2$ and  $G$ is well-edge-dominated.

Now, assume $G \in \cH_2$. Label the vertices of the $K_4$ as $v_1, v_2, v_3$ and $v_4$, remove the edge $v_1v_3$, and append leaves to vertex $v_1$. Let $D$ be a minimal edge dominating set of $G$. Using a similar argument to the one above we conclude that $G$ is well-edge-dominated.

\end{proof}

\begin{theorem}\label{thm:split}
A nontrivial, connected split graph $G$ is well-edge-dominated if and only if $G \in \{K_2, K_3, K_4, H_3\} \cup \cH_1 \cup \cH_2 \cup \{K_{1,n} : n \in \N\}$.
\end{theorem}

\begin{proof}
By Lemma~\ref{lem:WEDsplitgraphs}, each graph in $\{K_2, K_3, K_4, H_3\} \cup \cH_1 \cup \cH_2 \cup \{K_{1,n} : n \in \N\}$ is well-edge-dominated and is a split
graph by definition.

For the converse let $G$ be a connected, well-edge-dominated split graph.  We let $V(G) = K\cup I$ where $I$ is an independent set, $K=\{x_1,\ldots,x_k\}$, and $G[K]$ is a clique.
If $k = 1$, then $G = K_{1, n}$ where $n=|I|$. So we may assume $k \ge 2$. Suppose first that $I = \emptyset$. Thus, $G = K_k$, and $G \in \{K_2, K_3, K_4\}$ by Lemma~\ref{lem:Kn}. Therefore, we shall assume $I \ne \emptyset$. Assume first that $I = \{u\}$. If $u$ is adjacent to every vertex in $K$, then $G$ is clique.
So we shall assume $N(u) = \{x_1, \ldots, x_r\}$ where $r<k$.
Let $M = \{x_1x_2, x_3x_4, \dots, x_ix_{i+1}\}$ where $i+1= k$ if $k$ even or $i+1 = k-1$ if $k$ odd. We see that $M$ is a maximal matching in $G$ of size
$\left\lfloor \frac{k}{2}\right\rfloor$ and therefore a minimal edge dominating set. Let $M' = \{x_1x_2, x_1x_3, \dots, x_1x_{k-1}\}$. Note that $M'$ is
a minimal edge dominating set since $M' - \{x_1x_j\}$ for $2 \le j \le k-1$ does not dominate $x_jx_k$. However, $|M'| = k-2 \neq \left\lfloor \frac{k}{2}\right\rfloor = |M|$ when $k=2$ or $k \ge 5$. If $k=3$, then $r \in \{1, 2\}$.  In this case both $\{ux_1, x_2x_3\}$ and $\{x_1x_2\}$ are maximal matchings in $G$, which implies that $G$ is not well-edge-dominated.   Suppose then that $k=4$. If $r=1$, then $G \in \cH_1$.  If $r=2$, then
$G$ is not well-edge-dominated since $\{ux_1,x_1x_3,x_1x_4\}$ and $\{ux_1,x_2x_4\}$ are minimal edge dominating sets.  If $r=3$, then
$\{x_1x_2, x_3x_4\}$ and $\{ux_1, ux_2, ux_3\}$ are minimal edge dominating sets so $G$ is not well-edge-dominated.

Thus, we assume for the remainder of the proof that $|I| \ge 2$. We let $J = \{x \in K : N(x) \cap I \ne \emptyset\}$. Suppose first that $k=2$. If $J \ne K$,
then $G = K_{1, n}$ where $n-1=|I|$. Therefore, we shall assume $J = K$.  Without loss of generality, we may assume $x_1w_1 \in E(G)$ for some $w_1 \in I$.
If there exists $w_2 \in I - \{w_1\}$ such that $x_2w_2 \in E(G)$, then both $\{x_1x_2\}$ and $\{x_1w_1, x_2w_2\}$ are maximal matchings in $G$ and $G$ is not well-edge-dominated. So we may assume $N(x_2) \cap I = \{w_1\}$.  Since $G$ is connected and $|I| \ge 2$, it follows that there exists $w_2 \in I - \{w_1\}$
adjacent to $x_1$.  Thus, $\{x_1x_2\}$ and $\{x_1w_2, x_2w_1\}$ are both maximal matchings, which implies that $G$ is not well-edge-dominated. Having considered
all cases for $k=2$, we now assume $k \ge 3$.

Suppose there exist distinct vertices $x$ and $y$ in $J$ and distinct vertices $w_1$ and $w_2$, such that $w_1 \in N(x)\cap I$ and $w_2 \in N(y) \cap I$.  If $k$ is
even, then extend $xy$ to a maximal matching $M$ in $G[K]$.  Note that $M$ is a maximal matching in $G$.   If $k$ is odd and $J \neq K$, then let $z\in K-J$ and extend
$xy$ to a maximal matching $M$ of $G[K]$ such that $z\notin S(M)$.  Again, $M$ is a maximal matching in $G$.  In both of these cases let $M'=(M-\{xy\}) \cup \{xw_1,yw_2\}$.
Since $M'$ is also a maximal matching, $G$ is not equimatchable and thus also not well-edge-dominated, which is a contradiction.

Therefore, we shall assume that $k$ is odd and $K=J$. If there exists a $z \in K - \{x, y\}$ such that $N(z) \cap I \not\subseteq \{w_1, w_2\}$, then
extend $xy$ to a maximal matching $M$ of $G[K]$ such that $z \notin S(M)$.   Let $u \in (N(z) \cap I) - \{w_1, w_2\}$ and let $M' = M \cup \{uz\}$. In this case,
both $M'$ and $M'' = (M' - \{xy\}) \cup \{xw_1, yw_2\}$ are maximal matchings in $G$.  Therefore, $G$ is not equimatchable, which is a contradiction. Therefore, we shall assume for all $z \in K - \{x, y\}$, $N(z) \cap I \subseteq \{w_1, w_2\}$.

Suppose in addition  that $I = \{w_1, w_2\}$. Reindexing if necessary, we may assume that $\deg_G(w_1) \ge \deg_G(w_2)$ and we may assume  the vertices of $K$
have been enumerated in such a way that $N(w_1) \cap K = \{x_1, \dots, x_{\ell}\}$. If $\ell = k$, then we could instead partition $V(G)$ as $K' \cup I'$ where $I' = \{w_2\}$ and $K' = K \cup \{w_1\}$ and $K'$ induces a clique in $G$. Having already considered this case above, we instead assume $\ell \le k-1$. If $\ell = k-1$, then $x_k w_2 \in E(G)$ since $K=J$. Let $M' = \{w_1x_1, \dots, w_1x_{k-2}, w_2x_k\}$. Note that $M'$ is edge dominating since all vertices other than $x_{k-1}$ are covered by $M'$.
Moreover, $M'$ is minimal since $w_2x_k$ is its own private neighbor and $M' - \{w_1x_i\}$ does not dominate $x_ix_{k-1}$ for each $i \in[k-2]$.
On the other hand, $M = \{x_1x_2, x_3x_4, \dots, x_{k-2}x_{k-1}, w_2x_k\}$ is a maximal matching and therefore is a minimal edge dominating set.
Thus, $\frac{k+1}{2} = |M| =  |M'| = k-1 $, or equivalently, $k=3$. Since $\deg_G(w_1)=2 \ge \deg_G(w_2)$, the vertex $w_2$ is adjacent to at most one of
$x_1$ or $x_2$. If $\deg_G(w_2)=2$, then $G=H_3$.  On the other hand, $G \in \cH_2$ if $\deg_G(w_2)=1$.

Hence, we shall assume $\ell < k-1$. Let $M' = \{x_1x_2, x_1x_3, \dots, x_1x_{k-2}, x_kw_2\}$. The only vertices that are not covered by $M'$ are $x_{k-1}$ and
$w_1$. It follows that $M'$ edge dominates $G$ since $x_{k-1}w_1 \not\in E(G)$. Moreover, $M'$ is a minimal edge dominating set since $x_kw_2$ is its own private neighbor and $M' - \{x_1x_j\}$ does not dominate $x_jx_{k-1}$, for each $2 \le j \le k-2$. Now, since $M = \{x_1x_2, x_3x_4, \dots, x_{k-2}x_{k-1}, w_2x_k\}$
is a maximal matching, we get  $\frac{k+1}{2} = |M| = |M'| = k-2$, or equivalently $k=5$.
Notice that $\deg_G(w_1) \le 3$ since we have assumed $\ell < k-1$. On the other hand, $\deg_G(w_1) \ge 3$, for otherwise $\deg_G(w_2)>\deg_G(w_1)$ since $K=J$.
Therefore, $\deg_G(w_1) = 3$. Furthermore, $x_4w_2, x_5w_2 \in E(G)$, and it is possible that $w_2$ is adjacent to exactly one of $x_1,x_2$ or $x_3$, which does not affect
the following argument.  The set $M' = \{w_1x_1,w_1x_2,w_1x_3,w_2x_4\}$ covers all vertices of $G$ other than $x_5$, so it is edge dominating.  Moreover, since
any proper subset of $M'$ does not dominate some edge of the form $x_ix_5$, $M'$ is a minimal edge dominating set and we have a contradiction since
 $|M| = 3 < 4 = |M'|$.  Therefore, for the remainder of the proof we shall assume $|I|>2$.

Note that we are assuming for all $z\in K - \{x, y\}$, $N(z) \cap I \subseteq \{w_1, w_2\}$, $k$ is odd, and $K=J$. Without loss of generality, we may assume
the vertices of $K$ are enumerated in such a way that $x=x_1$ and $y=x_2$; in particular,  $x_1w_1 \in E(G)$ and $x_2w_2 \in E(G)$.
Furthermore, we may assume $x_2$ has a neighbor $w_3 \in I - \{w_1, w_2\}$. Assume first there exists $t \in K- \{x_1, x_2\}$ such that $tw_2 \in E(G)$. Reindexing if necessary, we may write $t = x_k$.
Let $M = \{x_1x_2, x_3x_4, \dots, x_{k-2}x_{k-1}, x_kw_2\}$.  Both $M$ and $M' = \{x_1w_1, x_2w_3, x_3x_4, x_5x_6, \ldots, x_{k-2}x_{k-1},x_kw_2\}$ are
maximal matchings in $G$.  However, $|M|= \frac{k+1}{2} < 3 + \frac{k-3}{2}=|M'|$, which contradicts the assumption that $G$ is well-edge-dominated.
Therefore, no such vertex $t\in K$ exists; that is, $N(x_i) \cap I = \{w_1\}$, for all $3 \le i \le k$.
Moreover, a similar argument (by interchanging $x_1$ and $x_3$) may be used to show that $x_1w_2 \not\in E(G)$.  This implies that  each vertex of $I - \{w_1\}$
is a leaf adjacent to  $x_2$.  Now the set $M'' = \{x_1x_3, x_1x_4, \dots, x_1x_k, x_2w_3\}$  is an edge dominating set since all vertices of $K$ are covered.
Since $x_2w_3$ is its own private neighbor with respect to $M''$ and $M'' - \{x_1x_j\}$ does not dominate $w_1x_j$, for  $3\le j \le k$, it follows that
$M''$ is a minimal edge dominating set of $G$.  On the other hand, $M = \{x_1x_2, x_3x_4, \dots, x_{k-2}x_{k-1}, x_kw_1\}$ is a maximal matching in $G$.
Since $G$ is a well-edge-dominated graph, we infer that $\frac{k+1}{2} =|M| = |M''| =  k-1$.  This implies that
$k=3$, and in this case $G \in \cH_2$.

Finally, we may assume there do not exist distinct vertices $x$ and $y$ in $J$ and distinct vertices $w_1$ and $w_2$, such that $w_1 \in N(x)\cap I$ and $w_2 \in N(y) \cap I$.  Since $|I| \ge 2$ and $G$ is connected, it follows that $|J|=1$.  Without loss of generality we assume $J=\{x_1\}$ and every vertex of $I$ is a leaf adjacent to $x_1$.
Let $M = \{x_1x_2, x_1x_3, \dots, x_1x_{k-1}\}$ and let $M' = \{x_1x_2, x_3x_4, \dots, x_{j}x_{j+1}\}$, where $j=k-1$ if $k$ is even and $j = k-2$ if $k$ is odd.
It is easy to see that both $M$ and $M'$ are minimal edge dominating sets, which implies that $k-2=|M| = |M'| = \left\lfloor \frac{k}{2}\right\rfloor$.  It follows
that $k \in \{3,4\}$.  If $k=4$, then $G\in \cH_1$.  On the other hand, $k=3$ gives a contradiction since it yields a graph obtained by attaching at least two leaves adjacent to
one fixed vertex of $K_3$.  This graph is not well-edge-dominated as can be easily shown.  This completes the proof.
\end{proof}

\section{Cartesian products} \label{sec:CartesianWED}

This section is devoted to proving our characterization of well-edge-dominated Cartesian products.

\vskip 5mm
\begin{lemma} \label{lem:CPnotWED}
Let $G$ and $H$ be nontrivial, connected graphs such that at least one of $G$ or $H$ has order at least $3$.
If $G$ has a perfect matching, then $G\Box H$ is not well-edge-dominated.
\end{lemma}

\begin{proof}
Suppose $G$ admits a perfect matching $M$ and suppose for the sake of contradiction that $G\Box H$ is well-edge-dominated.  By ``copying'' $M$ to each
$G$-fiber we see that $G \Box H$ also has a perfect matching. Suppose $G\Box H$ has order $2n$.  By Theorem~\ref{thm:rm}, it follows that $G\Box H = K_{2n}$ or
$G\Box H = K_{n,n}$.  This is a contradiction since no graph of order at least $6$  that is complete or complete bipartite is the Cartesian product of nontrivial
factors.
\end{proof}

\begin{prop}\label{prop:factorswed}
Let $G$ and $H$ be nontrivial connected graphs, neither of which has a perfect matching.  If $G\Box H$ is well-edge-dominated, then both $G$ and $H$ are well-edge-dominated.
\end{prop}

\begin{proof} Let $M= \{e_1, \dots, e_k\}$ be any maximal matching in $G$ where $e_i = x_iy_i$ for $i \in [k]$. Let $H$ be any connected graph with $V(H) = \{h_1, \dots, h_n\}$. Note that
\[P= \bigcup_{j=1}^n \bigcup_{i=1}^k \{(x_i, h_j)(y_i, h_j)\}\]
is a matching in $G\Box H$. By Lemma~\ref{lem:matchremove}, $G\Box H - N_e[P]$ is well-edge-dominated since $G\Box H$ is well-edge-dominated. Let $I = V(G) - S(M)$.
Since $M$ is a maximal matching that is not a perfect matching, the set $I$ is nonempty and independent.
Therefore, the nontrivial components of $G\Box H - N_e[P]$ are isomorphic to $H$. This implies $H$ is well-edge-dominated. Similarly, $G$ is well-edge-dominated.
\end{proof}

\begin{lemma} \label{lem:nopmnotWED}
If $G$ and $H$ are connected, nontrivial graphs neither of which has a perfect matching, then $G\Box H$ is not well-edge-dominated.
\end{lemma}

\begin{proof} Suppose to the contrary that there exist connected, nontrivial graphs $G$ and $H$, neither of which has a perfect matching, such that $G\Box H$ is well-edge-dominated. Thus, we may assume $n(G) \ge 3$ and $n(H) \ge 3$. By Proposition~\ref{prop:factorswed}, both $G$ and $H$ are well-edge-dominated. Let $r$ be the matching number of $G$ and $s$ be the matching number of $H$. Choose any maximal matching $M_1= \{e_1, \dots, e_s\}$ in $H$ and write $e_i = x_iy_i$ for $i \in [s]$. Let $I_H = V(H) - S(M_1)$.   Choose any maximal matching $M_G=\{f_1, \dots, f_r\}$ in $G$ and write $f_i = u_iv_i$ for $i \in [r]$. Let $I_G = V(G) - S(M_G)$. Let
\[P_1 = \bigcup_{g \in V(G)} \bigcup_{i=1}^s \{(g, x_i)(g, y_i)\}\]
and
\[P_2 = \bigcup_{h \in I_H}\bigcup_{i=1}^r \{(u_i, h)(v_i, h)\}.\]

Note that $P_1 \cup P_2$ is a maximal matching in $G\Box H$ since the only vertices in $G\Box H$ that are not covered by $P_1 \cup P_2$ are in $I_G \times I_H$, which is
an independent set in $G \Box H$.  Also, \[|P_1 \cup P_2| =  sn(G) + (n(H) - 2s)r.\]

Next, choose a maximal matching $M_2 = \{a_1, \dots, a_s\}$ in $H$ such that $S(M_1) \ne S(M_2)$ and write $a_i = w_iz_i$ for $i \in [s]$. Let $L = V(H) - (S(M_1)\cup S(M_2))$, $L'=S(M_1) - S(M_2)$, and $L'' = S(M_2) - S(M_1)$.  Choose any maximal independent set $J$ in $G$ and let  $N_1 = \{b_1c_1, \dots, b_tc_t\}$ be a minimal edge dominating set of $G - J$.

Let
\[Q_1 = \bigcup_{g \in V(G) - J}\bigcup_{i=1}^s \{(g, x_i)(g, y_i)\},\]

\[Q_2 = \bigcup_{g \in J}\bigcup_{i=1}^s \{(g, w_i)(g, z_i)\},\]
\[Q_3 = \bigcup_{h \in L}\bigcup_{i=1}^r \{(u_i, h)(v_i, h)\},\]
and
\[Q_4 = \bigcup_{h \in L''}\bigcup_{i=1}^t (\{(b_i, h)(c_i, h)\}.\]

We claim that $Q = \cup _{i=1}^4 Q_i$ is a minimal edge dominating set of $G\Box H$. If $g \in J$, then $Q_2 \cap E(^gH)$ is an edge dominating set of $^gH$. If $g \in V(G) - J$, then $Q_1 \cap E( ^gH)$ is an edge dominating set of $^gH$. Thus, for every $h_1h_2 \in E(H)$ and every $g \in V(G)$, the set $Q_1 \cup Q_2$ dominates the edge
$(g, h_1)(g, h_2)$.  Note that we can write $V(H)$ as a weak partition \[V(H) = L \cup L' \cup L'' \cup (S(M_1) \cap S(M_2)).\]

If $h \in L'$ and $g_1g_2 \in E(G)$, then $Q_1$ dominates the edge $(g_1, h)(g_2, h)$ since $J \times L'$ is independent and every vertex of $(V(G) - J) \times L'$ is covered by $Q_1$. If $h \in S(M_1) \cap S(M_2)$ and $g \in V(G)$, then $(g, h)$ is covered by $Q_1 \cup Q_2$. If $h \in L''$ and $g \in J$, then $(g, h)$ is covered by $Q_2$. On the other hand, if $h \in L''$ and $(g_1, h)(g_2, h)$ is an edge of $G\Box H$ where neither $g_1$ nor $g_2$ is in $J$, then $(g_1, h)(g_2, h)$ is dominated by $Q_4$. Finally, $Q_3 \cap E(G^h)$ is an edge dominating set of $G^h$,  for any $h \in L$ since $M_G$ is a maximal matching in $G$. Therefore, $Q$ is an edge dominating set of $G\Box H$.

Next, we show $Q$ is in fact a minimal edge dominating set. Let $e$ be an arbitrary edge in $Q$. If $e \in Q_1 \cup Q_2\cup Q_3$, then $Q - \{e\}$ does not dominate $e$. If $e \in Q_4$, say $e = (g_1, h)(g_2, h)$ where $h \in L''$, then some edge in the subgraph induced by $(V(G) - J) \times \{h\}$ is not dominated by $Q - \{e\}$ since $N_1$ is a minimal edge dominating set of $G-J$. Thus, $Q$ is a minimal edge dominating set of $G\Box H$.

Since $G\Box H$ is well-edge-dominated, $|P_1 \cup P_2| = |Q|$ where
\begin{eqnarray*}
|Q| &=& |Q_1| + |Q_2| + |Q_3| + |Q_4|\\
&=& (n(G) - |J|)s + |J|s + |L|r + |L''|t\\
&=& n(G)s + (n(H) - |S(M_1)| - |S(M_2)| + |S(M_1)\cap S(M_2)|)r\\
& &   +(2s- |S(M_1)\cap S(M_2)|)t\\
&=& n(G)s + (n(H) - 4s + |S(M_1)\cap S(M_2)|)r +(2s- |S(M_1)\cap S(M_2)|)t.
\end{eqnarray*}
In particular, this means
\[n(G)s + (n(H) - 2s)r = n(G)s + (n(H) - 4s + |S(M_1)\cap S(M_2)|)r +(2s- |S(M_1)\cap S(M_2)|)t\]
or equivalently
\[( |S(M_1)\cap S(M_2)|-2s)(r-t) = 0.\]
Note that $|S(M_1) \cap S(M_2)| \ne 2s$ since $S(M_1) \ne S(M_2)$. Thus, $r=t$ and every minimal edge dominating set of $G-J$ has cardinality $r$ since $N_1$ was chosen arbitrarily. It follows that $G-J$ is well-edge-dominated and $\gamma'(G-J)= \gamma'(G)$. Furthermore, we claim $G-J$ contains a perfect matching. Suppose to the contrary that $G-J$ does not admit a perfect matching. Let $A$ be any maximal matching of $G-J$ and let $x$ be a vertex of $G-J$ that is not covered by $A$. Since $J$ is a maximal independent set of $G$, there exists a vertex $y \in J$ where $xy \in E(G)$. However, $A \cup \{xy\}$ is a matching in $G$ of cardinality $r+1$, which is a contradiction.
Hence, $A$ is a perfect matching of $G-J$, and so $G-J$ is a well-edge-dominated, randomly matchable graph.  By Theorem~\ref{thm:wedrm}, $G -J = K_4$ or $G-J = K_{n, n}$
for some $n \ge 1$.

Notice that since $J$ is assumed to be a maximal independent set, each vertex of $G-J$ is adjacent to a vertex in $J$.
Let $e=wz$ be an arbitrary edge in $A$. If there exists a pair $u,v$ of distinct vertices in $J$ such that $u \in N(w) \cap J$ and $v \in N(z) \cap J$, then $(A - \{wz\}) \cup \{uw, vz\}$ is a  matching in $G$ of cardinality $r+1$, which is a contradiction.  It follows that $|N(w) \cap J|=1=|N(z) \cap J|$, and in fact $N(w) \cap J= N(z) \cap J$.  Since every edge of $G-J$
can be extended to a perfect matching of $G-J$ and since $G-J$ is connected, it follows that $|J|=1$ and therefore $G$ is a complete graph.  This implies that $G=K_5$, which
is not well-edge-dominated.  This contradiction completes the proof.
\end{proof}

Theorem~\ref{thm:cpWED} is restated
here for ease of reference.
\vskip3mm
\noindent \textbf{Theorem~\ref{thm:cpWED}} \emph{
If $G$ and $H$ are two connected, nontrivial graphs, then $G\Box H$ is well-edge-dominated if and only if $G\Box H = K_2 \Box K_2$.
}

\begin{proof}
The Cartesian product $K_2 \Box K_2$ is well-edge-dominated.  Conversely, suppose $G$ and $H$ are connected and nontrivial such that $G\Box H$ is well-edge-dominated.
By  Lemma~\ref{lem:nopmnotWED}, at least one of $G$ or $H$ has a perfect matching, and then by Lemma~\ref{lem:CPnotWED}  it follows that $G\Box H = K_2 \Box K_2$.
\end{proof}

\section{Open Questions}

In their study of connected, equimatchable graphs of girth at least $5$, Frendrup, Hartnell and Vestergaard~\cite{fhv-2010} characterized
the connected, well-edge-dominated graphs of girth at least $5$.  In particular, they proved the following result.

\begin{theorem} {\rm (\cite{fhv-2010})} \label{thm:fhvWED}
If $G$ is a connected graph with $g(G) \ge 5$, then $G$ is well-edge-dominated if and only if $G \in \{K_2,C_5,C_7\}$ or $G$ is bipartite with  partite sets $V_1$ and $V_2$
such that $V_1$ is the set of all support vertices of $G$.
\end{theorem}

In Theorem~\ref{thm:K3-freeWED} of this paper we showed that only one additional graph, namely $H^*$, is added to the list of connected, well-edge-dominated graphs
if the girth restriction is lowered to $4$ but we now require that the graph be nonbipartite.

A natural problem now presents itself.
\begin{problem}
Find a structural characterization of the class of connected, bipartite graphs of girth $4$ that are well-edge-dominated.
\end{problem}
By Theorem~\ref{thm:wedrm} this class contains $K_{n,n}$, for any $n\ge 2$ and by Theorem~\ref{thm:cpWED}
it does not contain any nontrivial Cartesian products other than $K_2 \Box K_2$.

For graphs that contain a triangle, we have characterized the connected, split graphs that are well-edge-dominated in Theorem~\ref{thm:split}.  Determining
the structure for arbitrary well-edge-dominated graphs of girth $3$ is an interesting problem.
\begin{problem}
Find a structural characterization of the class of connected graphs of girth $3$ that are well-edge-dominated.
\end{problem}

\vfill


\begin{thebibliography}{99}

\bibitem{av-1998} S.~Arumugam and S.~Velammal.  Edge domination in graphs.
    \textit{Taiwanese J. Math.},  {\bf 2(2)}:  173--179 (1998)
		
\bibitem{bfhmr-2020} J.~Baste, M.~F\"{u}rst, M.~A.~Henning, E.~Mohr and D.~Rautenbach. Domination versus edge domination.
   \textit{Discrete Appl. Math.}, {\bf 285}: 343--349  (2020)

\bibitem{bp-2008}  A.~Berger and O.~Parekh. Linear time algorithms for generalized edge dominating set problems.
    \textit{Algorithmica},  {\bf 50(2)}:  244--254  (2008)
	

\bibitem{bgo-2021} Y.~B\"{u}y\"{u}k\c{c}olak, D.~G\"{o}z\"{u}pek and S.~\"{O}zkan. Triangle-free equimatchable graphs.
   \textit{J. Graph Theory}, 1-22, to appear.

\bibitem{c-2010}  A.~Chaemchan. The edge domination number of connected graphs.
   \textit{Australas. J. Combin.},  {\bf 48}:  185--189  (2010)

\bibitem{fhn-1988} A.~Finbow, B.~Hartnell and R.~Nowakowski.  Well-dominated graphs: a collection of well-covered ones.
   \textit{Ars Comb.}, {\bf 25A}: 5--10 (1988)


\bibitem{fhv-2010} A.~Frendrup, B.~Hartnell and  P.~L.~Vestergaard. A note on equimatchable graphs.
   \textit{Australas. J. Combin.}, {\bf 46}:  185--190 (2010)

\bibitem{hk-1993}  J.~D.~Horton and K.~Kilakos.   Minimum edge dominating sets.
   \textit{SIAM J. Discrete Math.}, {\bf 6(3)}: 375--387 (1993)


\bibitem{hc-1995} S.~F.~Hwang and G.~J.~Chang.  The edge domination problem.
   \textit{Discuss. Math. Graph Theory}, {\bf 15(1)}: 51--57 (1995)
   	

\bibitem{ky-2015}   W.~F.~Klostermeyer and A.~Yeo.   Edge domination in grids.
     \textit{J. Combin. Math. Combin. Comput.}, {\bf 95}: 99--117  (2015)
		
\bibitem{lpp-1983} M.~Lesk, M.~D.~Plummer and W.~R.~Pulleyblank.  Equi-matchable graphs.
  \textit{Graph theory and combinatorics}, Academic Press, London, 239--254 (1984)

\bibitem{l-1974} M.~Lewin. Matching-perfect and cover-perfect graphs.
   \textit{Israel J. Math.}, {\bf 18}: 345--347 (1974)


\bibitem{m-1974} D.~H-C.~Meng. Matching and coverings for graphs.  Ph.~D.~Thesis,
        Michigan State University, East Lansing, MI. (1974)

\bibitem{mh-1977} S.~Mitchell  and S.~Hedetniemi.  Edge domination in trees.
   \textit{Congr. Numer.}, {\bf 19}: 489--509 (1977)

\bibitem{svk-2020} B.~Senthilkumar, Y.~B.~Venkatakrishnan and H.~N.~Kumar.  Domination and edge domination in trees.
   \textit{Ural Math. J.}, {\bf 6}: 147--152 (2020)

\bibitem{s-1979} D.~P.~Sumner.  Randomly matchable graphs.  \textit{J. Graph Theory}, {\bf 3(2)}: 183--186 (1979)

\bibitem{t-1993}  J.~Topp.  Graphs with unique minimum edge dominating sets and graphs with unique maximum independent sets of vertices.
    \textit{Discrete Math.}, {\bf 121}: 199--210 (1993)

\bibitem{yg-1980} M.~Yannakakis and F.~Gavril.  Edge dominating sets in graphs.
   \textit{SIAM J. Appl. Math.}, {\bf 38(3)}:  364--372  (1980)
	
\end{thebibliography}
\end{document}